\documentclass{amsart}
\usepackage{amsmath,amssymb}
\newcommand{\Alt}{\mathop{\mathrm{Alt}}}
\newcommand{\Sym}{\mathop{\mathrm{Sym}}}
\newcommand{\Aut}{\mathop{\mathrm{Aut}}}
\newcommand{\GL}{\mathop{\mathrm{GL}}}
\newcommand{\GU}{\mathop{\mathrm{GU}}}
\newcommand{\Sp}{\mathop{\mathrm{Sp}}}
\newcommand{\CSp}{\mathop{\mathrm{CSp}}}
\newcommand{\PSL}{\mathop{\mathrm{PSL}}}
\newcommand{\Hom}{\mathop{\mathrm{Hom}}}
\newcommand{\Ker}{\mathop{\mathrm{Ker}}}
\newcommand{\Gal}{\mathop{\mathrm{Gal}}}
\renewcommand{\Im}{\mathop{\mathrm{Im}}}
\renewcommand{\O}{\mathop{\mathrm{O}}}
\renewcommand{\wr}{\mathop{\mathrm{wr}}}
\newtheorem{theorem}{Theorem}[section]
\newtheorem{lemma}[theorem]{Lemma}
\newtheorem{corollary}[theorem]{Corollary}
\newtheorem{proposition}[theorem]{Proposition}
\newtheorem{conjB}[theorem]{Problem}
\newtheorem{example}[theorem]{Example}
\newtheorem{definition}[theorem]{Definition}
\def\cent#1#2{{\bf C}_{#1}(#2)}
\def\F#1{{\bf F}(#1)}
\begin{document}

\title[Coprime subdegrees]{
Coprime subdegrees for primitive permutation groups and completely
reducible linear groups}   

\author[S.~Dolfi]{Silvio Dolfi}
\address{Silvio Dolfi, Dipartimento di Matematica ``Ulisse
  Dini'', Universit\'a degli studi di Firenze,\newline Viale Morgagni 67/a, 50134 Firenze Italy}
\email{dolfi@math.unifi.it}

\author[R.~Guralnick]{Robert Guralnick}
\address{Robert Guralnick, Department of Mathematics, University of
  Southern California,\newline Los Angeles, California USA 90089--2532}
\email{guralnic@usc.edu}

\author[C. E. Praeger]{Cheryl E. Praeger}
\address{Cheryl E. Praeger, Centre for Mathematics of Symmetry and Computation,\newline
School of Mathematics and Statistics,\newline
The University of Western Australia,
 Crawley, WA 6009, Australia} \email{Cheryl.Praeger@uwa.edu.au}

\author[P. Spiga]{Pablo Spiga}
\address{Pablo Spiga, Centre for Mathematics of Symmetry and Computation,\newline
 School of Mathematics and Statistics,\newline
The University of Western Australia,
 Crawley, WA 6009, Australia} \email{pablo.spiga@uwa.edu.au}

\thanks{Address correspondence to P. Spiga,
E-mail: pablo.spiga@uwa.edu.au\\ 
The first author is supported by the MIUR project ``Teoria dei gruppi
ed applicazioni''. The second author was partially supported by NSF
grant DMS-1001962.    
The third author is supported by the ARC Federation
Fellowship Project 
FF0776186. The fourth author is supported by the University of Western Australia
as part of the  Federation Fellowship project.}

\subjclass[2000]{20B15, 20H30}
\keywords{coprime subdegrees; primitive permutation groups; completely reducible modules} 

\begin{abstract}
In this paper we answer a question of Gabriel Navarro about orbit sizes of
a finite linear group $H\subseteq \mathrm{GL}(V)$ acting completely reducibly on a
vector space $V$: if  the $H$-orbits containing the vectors $a$ and
$b$ have coprime lengths $m$ and $n$, we prove that the $H$-orbit
containing $a+b$ has length 
$mn$. Such groups $H$ are always reducible if $n,m>1$. In fact, if $H$ is an
irreducible linear group, we show that, for every pair of non-zero vectors,
their orbit lengths have a non-trivial common factor.

In the more general context of finite primitive permutation groups
$G$, we show that 
coprime non-identity subdegrees are possible if and only if $G$ is of
O'Nan-Scott 
type AS, PA or TW. In a forthcoming paper 
we will show that, for a
finite primitive permutation group, a set of pairwise coprime
subdegrees has size at most $2$.
Finally, as an application of our results, we prove that a field  has at most $2$ finite extensions of pairwise coprime indices with the same normal closure. 
\end{abstract}
\maketitle

\section{Introduction}\label{introduction}
\subsection{Completely reducible linear groups}\label{sub:1}
In this paper we are concerned with the orbit lengths of a \emph{completely
reducible} linear group and with the \emph{subdegrees} of a primitive
permutation group. Given a field $k$, a $kH$-module $V$ is said to be completely
reducible if $V$ is a direct sum of irreducible
$kH$-modules. Furthermore, the set of  subdegrees of a finite transitive
permutation group $G$ is the set of orbit lengths of  the stabilizer
$G_\omega$ of a point $\omega$.  
  
Our first main result is a positive answer to a question of Gabriel
Navarro~\cite{GN} about actions of a finite linear group. Indeed,
Navarro asked whether a finite completely reducible $kH$-module $V$
with two $H$-orbits of relatively prime lengths $m$ and $n$ 
has an orbit of size $mn$.

\begin{theorem}\label{thm1}Let $k$ be a field, let $H$ be a finite group, let $V$ be a completely reducible
  $kH$-module and let $a$ and $b$ be elements of $V$. If  the $H$-orbits 
  $a^H$ and $b^H$  have sizes $m$ and $n$, and
$m,n$ are  relatively prime, then $\cent H {a+b}=\cent Ha \cap \cent Hb$ and
the $H$-orbit $(a+b)^H$ has size $mn$.
\end{theorem}

We note that Theorem~\ref{thm1} explicitly exhibits an $H$-orbit of size
$mn$, namely the orbit containing $a+b$. Furthermore,
Example~\ref{Exa1} shows that the ``completely reducible'' hypothesis 
in Theorem~\ref{thm1} is essential. Martin Isaacs~\cite[Theorem]{Isaacs3} has proved a similar result 
under stronger arithmetical conditions on $m$, $n$ and the characteristic of $k$.

Our second main result is a somehow remarkable theorem  (in our opinion)
on irreducible linear groups. Theorem~\ref{thm2} shows that
the groups arising in Theorem~\ref{thm1} are always reducible if $n,m>1$.

\begin{theorem}\label{thm2}Let $k$ be a field, let $H$ be a finite group, let $V$ be a non-trivial irreducible
  $kH$-module and let
  $a$ and $b$ be in $V\setminus\{0\}$. Then the sizes of the
  $H$-orbits $a^H$ and  $b^H$ have a non-trivial common factor.  
\end{theorem}
The strategy for proving Theorem~\ref{thm2} is to reduce the problem 
inductively to the case where $H$ is a non-abelian simple group
which admits a maximal 
factorization 
$H=AB$ with $|H:A|$ relatively prime to $|H:B|$. Table~\ref{table}
contains all such triples $(H,A,B)$. In the case where $H$ is a
sporadic simple group, we use the
information in Table~\ref{table} for the proof of
Theorem~\ref{thm2}. Furthermore, we observe that as a consequence of
Theorem~\ref{thm2}, if $(H,A,B)$ is one of the triples in
Table~\ref{table}, then there are no irreducible representations of
$H$ with $A$ and $B$ vector stabilizers.     

A direct application of Theorem~\ref{thm2} gives the following corollary.

\begin{corollary}\label{cornew}Let $k$ be a field, let $H$ be a finite group, let $V$ be a non-trivial finite dimensional $kH$-module  and let $a$ and $b$ be elements of $V$. If both $H$-orbits $a^H$ and $b^H$ span $V$, then $|a^H|$ and $|b^H|$ have a non-trivial common factor. 
\end{corollary} 

In the same direction as Corollary~\ref{cornew}, in the case of $p$-soluble groups, we prove the following theorem.

\begin{theorem}\label{thm:101}
Let $k$ be a field of characteristic $p\ge 0$, $H$ a
$p$-soluble finite group, $V$ a  $kH$-module and $a\in V$ fixed by a Sylow
$p$-subgroup of $H$ and with the $H$-orbit
$a^H$ spanning $V$. Then
\begin{enumerate}
\item[(a)] $\dim \cent VH \le 1$; and
\item[(b)] if $b\in V$ and  $\gcd(|a^H|,|b^H|)=1$, then $b \in C_V(H)$.
\end{enumerate}
\end{theorem}

Here, by abuse of notation, ``$0$-soluble finite group'' means ``finite group'' and  a ``Sylow $0$-subgroup" is the ``identity subgroup''. In the proof of Theorem~\ref{thm:101} we do not make use  of the Classification of the Finite Simple Groups. Moreover, since in an irreducible $kH$-module $V$ every non-trivial $H$-orbit spans $V$, we see that Theorem~\ref{thm:101}~(b) generalizes (for the class of $p$-soluble groups) Theorem~\ref{thm2} and, in particular, offers an independent and more elementary proof. Note that if $p$ does not divide the order of $H$ (including the case $p=0$),  then $H$ is $p$-soluble, the Sylow $p$-subgroups of $H$ are trivial and, in particular, Theorem~\ref{thm:101} applies in this situation.

\subsection{Coprime subdegrees in primitive permutation groups}\label{sub:2}
In the more general context of finite primitive permutation groups
$G$, we investigate 
coprime subdegrees according to the O'Nan-Scott type of $G$. (We say
that a subdegree $d$ of $G$ is \emph{non-trivial} if $d\neq 1$.) In
particular, Theorem~\ref{thm2} yields that the primitive permutation
group $G=V\rtimes H$ acting on $V$ has no pair of non-trivial coprime
subdegrees. One of the most important modern methods for analyzing a
finite 
primitive permutation group $G$ is to study the \emph{socle} $N$ of
$G$, that is, the subgroup generated by the minimal normal subgroups
of $G$. The socle of an arbitrary finite group is isomorphic to a direct product of simple groups, and, for finite primitive
groups these simple groups are pairwise isomorphic. The O'Nan-Scott
theorem describes in detail the embedding of $N$ in $G$ and collects
some useful information on the action of $N$. In~\cite{Pr} eight types
of primitive groups are defined (depending on the structure and on the
action of the socle), namely HA (\textit{Holomorphic Abelian}), AS
(\textit{Almost Simple}), SD
(\textit{Simple Diagonal}), CD
(\textit{Compound Diagonal}), HS
(\textit{Holomorphic Simple}), HC
(\textit{Holomorphic Compound}), TW 
(\textit{Twisted wreath}), PA
(\textit{Product Action}), and it is shown in~\cite{LPS1} that every primitive group
belongs to exactly one of these types. 

\begin{theorem}\label{thm3}
Let $G$ be a finite primitive permutation group. If $G$ has two non-trivial coprime
subdegrees, then $G$ is of AS, PA or TW type. Moreover, for each
of the O'Nan-Scott types AS, PA and TW, there exists a primitive group of
this type with
two non-trivial coprime subdegrees. 
\end{theorem}

It is possible, for a single primitive group to have several different pairs of non-trivial coprime subdegrees. We give a construction of groups of PA type with this property in Example~\ref{exa3}. However, in the case of primitive groups of TW type, it is not possible to have as many as three pairwise coprime non-trivial subdegrees.

\begin{theorem}\label{TW}
For a finite primitive permutation group of TW type, the maximal size of a set of pairwise coprime non-trivial subdegrees is at most $2$. 
\end{theorem}

Using the Classification of the Finite Simple
Groups, we have proved  the following theorem in~\cite{DGPS2}.  

\begin{theorem}\label{conj1}
Let $G$ be a finite primitive permutation group. The maximal size of a set of 
pairwise coprime non-trivial subdegrees of $G$ is at most $2$.
\end{theorem}

Theorem~\ref{conj1} is related to a result on primitive
groups first observed by Peter Neumann to be a consequence of a 1935 theorem of Marie Weiss. Its statement~\cite[Corollary~2, p.~93]{Piter} is: if a
primitive group has $k$ 
pairwise coprime non-trivial subdegrees, then its  rank is at least $2^k$. Theorem~\ref{conj1} shows that this result  can
only be applied with $k=1$ or $k=2$. In light of Theorems~\ref{thm3} and~\ref{TW}, proving Theorem~\ref{conj1} reduces to consideration of 
primitive permutation groups of AS and PA type.
In this paper we show that a proof of Theorem~\ref{conj1} reduces to  a similar problem for transitive nonabelian simple permutation groups (it is this  reduction that is used in~\cite{DGPS2} to prove Theorem~\ref{conj1}). 
   
\begin{definition}\label{def}{\rm Let $T$ be a nonabelian simple group
    and $L$ a subgroup of $T$. We say that $L$ is {\em pseudo-maximal}
    in $T$ if there exists an almost simple group $H$ with socle $T$
    and a maximal subgroup $M$ of $H$ with  
$T\nsubseteq M$ and $L=T\cap M$.} 
\end{definition}

We announce here a proof of  the following theorem about nonabelian simple groups (which again will be proved in~\cite{DGPS2}).

\begin{theorem}\label{conj3}Let $T$ be a  transitive nonabelian simple permutation group and assume that the stabilizer of a point is pseudo-maximal in $T$. Then the maximal size of a set of pairwise coprime  non-trivial subdegrees of $T$ is at most $2$.
\end{theorem}

Since a pseudo-maximal subgroup of $T$ is not necessarily a maximal subgroup, we see that Theorem~\ref{conj3} is formally stronger than Theorem~\ref{conj1} for the class of nonabelian simple permutation groups. However, we prove in this paper that these two theorems are strongly related. 

\begin{theorem}\label{thm4}Theorem~\ref{conj1} follows from Theorem~\ref{conj3}.
\end{theorem}

We conclude with a problem on relatively prime subdegrees in primitive groups.
\begin{conjB}\label{conj2}
Determine the finite primitive permutation groups $G$ having two non-trivial coprime subdegrees $m$
and $n$ for which $mn$ is not a subdegree of $G$.
\end{conjB}
This problem is related to another classical result due to
Marie Weiss~\cite[Theorem~3, p.~92]{Piter}: if $m$ and $n$ are non-trivial
coprime subdegrees of a primitive group $G$ and $m<n$, then $G$
has a subdegree $d$ such that $d$ divides $mn$ and
$d>n$. In Problem~\ref{conj2} we suggest that (apart a small list of exceptions) $d$ can be
chosen to be $mn$. Actually, we know only one almost simple group $G$ where $d$ cannot be taken to be $nm$. 

\begin{example}\label{HS}{\rm
Let $G$ be the sporadic simple group $HS$ in its primitive permutation representation of 
degree $3850$ on the cosets of $2^4.\Sym(6)$. Using the computational algebra system~\texttt{magma}~\cite{magma}, it is easy to check that 
the subdegrees of $G$ are 
$1, 15, 32, 90, 120, 160, 192, 240,240, 360, 960, 1440$. 
In particular, we see that $15$ and $32$ are coprime but there is no subdegree of size $15\times 32$=$480$.}
\end{example}

We are grateful to Michael Giudici for providing this beutiful example.

We now give some examples which demonstrate 
that Theorem~\ref{conj1} is false for 
transitive groups that are not 
primitive.   These examples show that there is no upper bound
on the number of pairwise coprime non-trivial subdegrees for general transitive groups.  

\begin{example}\label{Frob}{\rm 
Let $G$ be
    the direct product $F_1\times \cdots \times F_\ell$ of $\ell$
    Frobenius groups. For each $i\in \{1,\ldots,\ell\}$, let
    $N_i$ be the Frobenius kernel of $F_i$ and let $K_i$ be a Frobenius
    complement for $N_i$ in $F_i$.  Assume that $|K_i|$ is coprime to 
    $|K_j|$, for every two distinct elements $i$ and $j$ in
    $\{1,\ldots,\ell\}$. Write $N=N_1\times 
    \cdots \times N_\ell$ and $K=K_1\times \cdots \times
    K_\ell$. Clearly, the group $G$ acts on $N$ as a holomorphic
  permutation  group, that is, $N$ acts on $N$ by right multiplication and $K$
    acts on $N$ by group conjugation. The stabilizer in $G$ of the
    element $1$ of $N$ is $K$. Now, for each $i\in
    \{1,\ldots,\ell\}$, let $n_i\in N_i\setminus\{1\}$ and let
    $\omega_i=(1,\ldots,1,n_i,1,\ldots,1)$ be the element of $N$ with
    $n_i$ in the $i^{\mathrm{th}}$ coordinate and $1$ everywhere
    else. Clearly, $|\omega_i^{K}|=|K:\cent
    K{\omega_i}|=|K_i|$. Therefore $G$ has a set of at least $\ell$ pairwise coprime non-trivial subdegrees.}
\end{example}

However, if we restrict to \emph{faithful subdegrees} of a transitive group $G$, that is, subdegrees $d$ such that there exists an orbit  of length $d$ of a stabilizer $G_\alpha$ on which $G_\alpha$ acts faithfully, then in fact we can show that a conclusion analogous to the statement of Theorem~\ref{conj1} does hold.  We note that, in particular, every primitive permutation group has a faithful subdegree~\cite[Theorem~3]{Bob}.

\begin{theorem}\label{cor:3}
Let $G$ be a finite transitive permutation group of degree $n > 1$.
Assume that $G$ is not regular and let $H$ be the stabilizer of a point. Then a set of faithful subdegrees that are pairwise coprime has size at most $2$. 
Moreover, if the Fitting subgroup of $H$ is non-trivial, then any two faithful subdegrees of $G$ have a non-trivial common factor.
\end{theorem}

The proof of Theorem~\ref{cor:3} also yields the following result about field
extensions.

\begin{theorem}\label{cor:4}
Let $k$ be a field and let $k_1,\ldots,k_t$ be finite extensions of $k$ all with
the same normal closure $K$. Assume that the indices $[k_i : k]$ are pairwise coprime. Then $t \leq 2$.
\end{theorem}

\subsection{Structure of the paper}\label{sub:4}The structure of this paper is straightforward: we prove Theorem~\ref{thm2}, Corollary~\ref{cornew} and Theorem~\ref{thm:101}
in Section~\ref{secthm2}; we prove Theorem~\ref{thm1} in
Section~\ref{secthm1}; we prove Theorem~\ref{thm3} in
Section~\ref{secthm3}; we prove Theorems~\ref{TW} and~\ref{thm4} in
Section~\ref{sec4}; we prove Theorems~\ref{cor:3} and~\ref{cor:4} in Section~\ref{CorProof}; and we give Table~\ref{table} in Section~\ref{sectable}.

\section{Proofs of Theorems~\ref{thm2} and~\ref{thm:101} and Corollary~\ref{cornew}}\label{secthm2}

We say that a factorization $H=AB$ is
\emph{coprime} if $|H:A|$ is relatively prime to $|H:B|$ and both
$A,B$ are proper subgroups of $H$. Also $H=AB$ is
\emph{maximal} if $A$ and $B$ are maximal subgroups of $H$. We start
by proving a preliminary theorem on finite classical groups.  We let
$\tau$ denote the transpose inverse map of $\GL_n(q)$, that is, 
$x^\tau=(x^{tr})^{-1}$ where $x^{tr}$ is the transpose matrix of
$x$. We denote by $\CSp(2n,q)$ the conformal
        symplectic group, that is, the elements of $\GL_{2n}(q)$
        preserving a given symplectic form up to a scalar multiple.

\begin{theorem}\label{thm8}Let $n\geq 2$.
\begin{description}
\item[(a)]Every element of $\GL_n(q)$ is conjugate to its inverse in
  $\GL_n(q)\langle\tau\rangle$. 
\item[(b)]Every element of $\GU_n(q)$ is conjugate to its inverse in $\GU_n(q)\langle \tau\rangle$.
\item[(c)]Every element of  $\Sp(2n,q)$ is conjugate to its inverse in
  $\CSp(2n,q)$.
\item[(d)]Every element of  $\O^{\epsilon}(n,q)$ is conjugate to its
  inverse in $\O^{\epsilon}(n,q)$, for $\epsilon\in \{\pm,\circ\}$.
\end{description}
\end{theorem}
\begin{proof}
We prove \textbf{(a)} and \textbf{(b)} first.  Let $X=\GL_n(k)$ be
the algebraic group obtained by taking the algebraic closure $k$ of
the finite field $\mathbb{F}_q$.  
Let $F:X\to X$ be the Lang-Steinberg map obtained by raising each
entry of a matrix $x$ of $X$ to the $q$th power, and  $G:X\to X$ the
Lang-Steinberg map $F\circ \tau$. As usual, we denote by $X^F$ and by
$X^G$ the fixed points of $F$ and of $G$. In our 
case, we have $X^F=\GL_n(q)$ and $X^G=\GU_n(q)$. 
  Let $x$ be in $X^F$.  Then
$x^{tr}$ and $x$ are clearly conjugate in the algebraic group
  $X$ and hence also $x^{\tau}=(x^{tr})^{-1}$ and $x^{-1}$ are conjugate in $X$.
  Since the centralizer of any 
element of $X$ is connected, 
it follows by the Lang-Steinberg theorem that $x^\tau$ and $x^{-1}$
are  conjugate in $X^F$. Therefore $x$ and $x^{-1}$ are   
conjugate in $\GL_n(q)\langle\tau\rangle$ and~\textbf{(a)} is
proved. Now, let  $x$ be in $X^G$. As we have noted in the proof
of~\textbf{(a)}, the elements $x^\tau$ and $x^{-1}$ are conjugate in
the algebraic group $X$. It follows by the
Lang-Steinberg theorem that $x^\tau$ and $x^{-1}$ 
are  conjugate in $X^G$. Therefore $x$ and $x^{-1}$ are   
conjugate in $\GU_n(q)\langle\tau\rangle$ and~\textbf{(b)} is
proved.

\textbf{(c)}   and~\textbf{(d)},  when $q$ is even, are the main
theorem of~\cite{Gow}. Finally, \textbf{(c)} and~\textbf{(d)}, when
$q$ is odd, are proved in~\cite{Won}. 
\end{proof}
Given a field $k$ and a $kH$-module $V$, we let $V^*=\Hom_k(V,k)$ denote the {\em
  dual} $kH$-module of $V$. Furthermore, we denote by $V_A$ the {\em
  restriction} of $V$ to the subgroup $A$ of $H$. Finally, if $A$ is a
subgroup of $H$ and if $V$ is a $kA$-module, then we denote by
$V_A^H=V\otimes_{kA} kH$ the module {\em induced} by $V$ from $A$ to
$H$.  

\begin{lemma}\label{lemma1}   Suppose that $H = AB$ is a factorization.  If $V$ is a non-trivial
irreducible $kH$-module, then either $A$ fixes no element of $V\setminus\{0\}$ or $B$ fixes no element of  $V^*\setminus\{0\}$. 
\end{lemma}
\begin{proof}
We argue by contradiction and we assume that $A$ fixes $a\in V\setminus\{0\}$ and that $B$ fixes $b\in V^*\setminus\{0\}$. 

Let $\Omega$ (respectively $\Delta$) be the set of right cosets of $A$
(respectively $B$) in $H$. Clearly, $H$ acts transitively on $\Omega$
and $\Delta$, and as $H=AB$, the group $B$ is transitive on $\Delta$
and $A$ is transitive on $\Omega$. Let $k_A^H$ (respectively $k_B^H$)
be the permutation module for the action of $H$ on $\Omega$
(respectively $\Delta$). Since $A$ is transitive on $\Delta$, the
multiplicity of the trivial $kA$-module $k$ in $(k_B^H)_A$ is $1$, that
is, $\dim \Hom_{kA}(k,(k_B^H)_A)=1$. From Frobenius reciprocity,
it follows that $\dim
\Hom_{kH}(k_A^H,k_B^H)=\dim\Hom_{kA}(k,(k_B^H)_A)=1$.  Therefore, the
only $H$-homomorphism of $k_A^H$ to $k_B^H$ is the homomorphism
$\varphi$ with $\Ker \varphi$ of codimension $1$ in $k_A^H$ and with
$\Im\varphi$ the trivial submodule of $k_B^H$. 

Since $A$ fixes the non-zero vector $a$ of $V$, we have $0\neq
\Hom_{kA}(k,V_A)\cong \Hom_{kH}(k_A^H,V)$ and hence $k_A^H$ has a
homomorphic image isomorphic to $V$. Similarly,  
since $B$ fixes the non-zero vector $b$ of $V^*$, we have $0\neq
\Hom_{kB}(k,V^*_B)\cong \Hom_{kH}(k_B^H,V^*)$ and hence $k_B^H$ has a
homomorphic image isomorphic to $V^*$. Using duality and the fact that
$k^*\cong k$, we obtain  that $(k_B^H)^*\cong (k^*)_B^H\cong k_B^H$
has a submodule isomorphic to $V^{**}\cong V$. This shows that there
exists an $H$-homomorphism $\psi:k_A^H\to k_B^H$ with
$k_A^H/\Ker\psi\cong V$ and with $\Im\psi\cong V$.  
Since $V$ is non-trivial, we obtain that
$\dim\Hom_{kH}(k_A^H,k_B^H)>1$, a contradiction.  
\end{proof}
Here we say that a factorization $H=AB$ is {\em exact} if $A\cap B=1$.

\begin{lemma}\label{exact}Suppose that $H = AB$ is a coprime exact
  factorization. If $V$ is a non-trivial irreducible $kH$-module, then
  either $A$ or $B$ fixes no element of $V\setminus\{0\}$. 
\end{lemma}
\begin{proof}
We argue by contradiction and we assume that both $A$ and $B$ fix some
non-zero vector of $V$. 
Let $r$ be the  characteristic of the field $k$. Since $|H:A|$ is
relatively prime to $|H:B|$ and $A\cap B=1$, we have that either $r$
does not divide $|A|$ or $r$ does not divide $|B|$. Replacing $A$ with
$B$ if necessary, we may assume that $r$ does not divide $|B|$.  Since
the characteristic of $k$ is coprime to the order of $B$, the module
$V_B$ is a 
completely reducible $kB$-module. Therefore,
$V_B=W_1\oplus\cdots\oplus W_s$ where $W_i$ is an  irreducible
$kB$-module, for each $i\in \{1,\ldots,s\}$. Since $B$ fixes a non-zero
vector of $V$, we have that, for some $i\in \{1,\ldots,s\}$, $W_i$ is
a trivial $kB$-module. Now, $V_B^*=W_1^*\oplus \cdots \oplus W_s^*$
and hence $W_i^*$ is a trivial submodule of $V_B^*$. This shows that
$B$ fixes a non-zero vector of $V^*$, but this contradicts
Lemma~\ref{lemma1}.  
\end{proof}

The following lemma is Lemma~$5.1$ in~\cite{GT}.
\begin{lemma}\label{GTlemma} Suppose that every element of $H$ is
  conjugate to its inverse via an element of $\Aut(H)$. If  $V$ is an
  irreducible $kH$-module, then $V^*\cong V^x$ for some $x\in
  \Aut(H)$. 
\end{lemma}
\begin{proof}
Write $G=H\rtimes \Aut(H)$. We can view $H$ as a  subgroup of  $G$. 
Since $H$ is normal in $G$, from~\cite[Theorem~$8.6$]{GN1} we see
that  the module $M=(V_H^G)_H$ is completely reducible with irreducible
summands $V^x$, for $x\in G$. Furthermore, since  
every element of $H$ is conjugate to its inverse via an element of
$G$, we obtain that the Brauer character of $M$ is real valued. Now,
from~\cite[Theorem~$1.19$ and Lemma~$2.2$]{GN1}, we see that 
completely reducible modules with real Brauer characters are self dual
and hence  $M$ is self dual, that is, $M^*\cong M$. Hence $V^*$ is an
irreducible direct summand of $M$, and so $V^* \cong V^x$ for some $x\in G$. 
\end{proof}

\begin{lemma}\label{idea} Suppose that every element of $H$ is
  conjugate to its inverse via an element of $\Aut(H)$. If $H = AB$ is
  a coprime factorization and $V$ is a non-trivial irreducible $kH$-module, then
  either $A$ fixes no element of $V\setminus\{0\}$ or $B$ fixes no
  element of $V\setminus\{0\}$. 
\end{lemma}
\begin{proof}
From Lemma~\ref{GTlemma}, $V^*\cong V^x$ for some $x\in \Aut(H)$. As
$H=AB$ is a coprime factorization, we obtain $H = AB^x$. 

We argue by contradiction and we assume that $A$ fixes the non-zero
vector $a$ of $V$ and that $B$ fixes the non-zero vector $b$ of $V$.
So  $B^x$ fixes the vector $b^x$ of $V^x$ and, as $V^x\cong V^*$, the
group $B^x$ fixes some non-zero vector of $V^*$. This contradicts
Lemma~\ref{lemma1} applied to $G=AB^x$, and so the lemma is proven. 
\end{proof}

In the following proposition we prove Theorem~\ref{thm2} in the case
that the group $H$ is a non-abelian simple group. 
\begin{proposition}\label{thmHsimple}
Let $H$ be a non-abelian simple group, $V$ be a non-trivial
irreducible $kH$-module, and $a$ and $b$ be in $V\setminus\{0\}$. Then
the sizes of the $H$-orbits $a^H$ and $b^H$ have a non-trivial common
factor.  
\end{proposition}
\begin{proof}
We argue by contradiction and we assume that $a^H$ and $b^H$ have
relatively prime sizes. Since $|a^H|=|H:\cent H a|$ and $|b^H|=|H:\cent
H b|$ are coprime, $H=\cent H a\cent H b$ is a coprime
factorization. Now we use the classification of the finite simple
groups. 

If $H$ is a classical group, we see from Theorem~\ref{thm8} that every
element of $H$ is conjugate to its inverse via an element of
$\Aut(H)$. Clearly, the same result holds true if $H$ is an
alternating group. Therefore, if $H$ is a classical group or an
alternating group, we obtain a contradiction  from Lemma~\ref{idea}
(applied with $A=\cent H a$ and $B=\cent H b$). This shows that $H$ is
either an exceptional group of Lie type or a sporadic simple
group. From Table~\ref{table}, we see that exceptional groups of Lie
type do not admit coprime factorizations. Therefore, $H$ is a sporadic
simple group. Again, using Table~\ref{table}, we see that the only
sporadic simple groups admitting a coprime factorization are $M_{11},
M_{23}$ and $M_{24}$. In the rest of this proof we consider separately
each of these groups. Note that Table~\ref{table} determines all
possible coprime factorizations $H=AB$ with $A$ and $B$ maximal in
$H$.

\smallskip

\noindent\textsc{Case $H=M_{11}$. }We first consider the case that
$\cent H a \subseteq A=L_2(11)$ and $\cent H b \subseteq B=M_{10}$. We have $|H:A|=12$,
$|H:B|=11$ and $A\cap B\cong \Alt(5)$. 
As $2$ and $3$ divide $|H:A|$ and
as $B\cong \Alt(6).2$ has no subgroups of index $5$, in order to have
$\gcd(|H:A|,|H:\cent H b |)=1$, we must have $B=\cent H b$. Since $|H:C_H(b)|=|H:B|=11$
and $|H:\cent H a|$ is coprime to $11$,
the group $\cent H a$ has order divisible by $11$ and hence it contains a
Sylow $11$-subgroup $S$. Now we have
$H=S\cent H b$ with $S\cap \cent H b=1$ and hence the result follows from Lemma~\ref{exact}.

Now we consider the case that $\cent H a\subseteq A=L_2(11)$ and $\cent H b\subseteq B=
M_9.2$. We have $|H:A|=12$, $|H:B|=55$ and $A\cap B\cong \Alt(4)$. 
Since $|A\cap B|=|H:A|=12$, by coprimality we have
$B=\cent H b$. Therefore, $|H:\cent H b|=55$ and $55$ divides
$|\cent H a|$. From the subgroup structure of $A=L_2(11)$, we see that
$\cent H a$ contains a
subgroup $S$ of order $55$. In particular, 
$H=S\cent H b$ and $S\cap \cent H b=1$, and the result follows from Lemma~\ref{exact}.

\smallskip

\noindent\textsc{Case $H=M_{23}$. }In this case we have three maximal
factorizations to consider. We start by studying the case that 
$\cent H a\subseteq A=M_{22}$ and $\cent H b\subseteq B=23:11$. As $|H:A|=23$ and $\gcd(|H:\cent H a|,|H:\cent H b|)=1$, we have that $23$
divides $\cent H b$. Let $S$ be a Sylow $23$-subgroup of $\cent H b$. Now,
$H=\cent H a S$ and $\cent H a\cap S=1$, and the result follows as usual from Lemma~\ref{exact}. 

The other two maximal coprime factorizations of $M_{23}$ in
Table~\ref{table} are exact and hence the result follows again from Lemma~\ref{exact}.

\smallskip

\noindent\textsc{Case $H=M_{24}$. }We have $\cent H a\subseteq A=M_{23}$,
$\cent H b\subseteq B=2^6.3.\Sym(6)$, $|H:A|=24$, $|H:B|=1771=7\cdot 11\cdot 23$ and $|A\cap
B|=5760=2^7\cdot 3^2\cdot 5$. 
Since $|H:\cent H a|$ is
divisible by $2$ and $3$, and $|H:\cent H b|$ is relatively prime to
$|H:\cent H a|$, we 
have that $\cent H b$ contains a Sylow $2$-subgroup and a Sylow
$3$-subgroup of $H$. Thus $|\cent H b |$ is divisible by
$2^{10}\cdot 3^3$ and $|B:\cent H b|\leq 5$. Since $B$ has no subgroup
of index $5$, we obtain $B=\cent H b$. With a similar argument applied to
$\cent H a$, we get that $7\cdot 11\cdot 23$ divides $|\cent H
a|$. From~\cite{ATLAS}, 
we see that $M_{23}$ has no proper subgroup of order divisible by $7,11$ and
$23$. Therefore $A=\cent H a$.

Let $M$ be the permutation module $k_A^H$. (Thus $M$ is the
permutation module of the $2$-transitive action of $H$ on a set
$\Omega$ of size $24$. In particular, $M$ is one of the modules
investigated by Mortimer in~\cite{Mortimer}.) Since $A$ fixes the
non-zero vector $a$ of $V$, we have $\Hom_{kA}(k,V_A)\neq 0$ and so, from
Frobenius reciprocity, we obtain $\Hom_{kH}(M,V)\neq 0$. Hence the
$kH$-module $V$ is  isomorphic to $M/W$, for some maximal
$kH$-submodule $W$ of 
$M$. Let $(e_{\omega})_{\omega\in \Omega}$ be the canonical basis of
$M$ and let $p$ be the characteristic of $k$. 

Let $e=\sum_{\omega\in
  \Omega}e_\omega$, $\mathcal{C}=\langle e\rangle$ and
$\mathcal{C}^\perp=\{\sum_{\omega\in \Omega}c_\omega 
e_\omega\mid \sum_{\omega\in\Omega}c_\omega=0\}$. Clearly,
$\mathcal{C}$ and $\mathcal{C}^\perp$ are submodules of $M$. Assume
that $p\neq 2,3$.  From~\cite[Table~$1$]{Mortimer}, we see that  the
module $M$ is completely reducible,
$M=\mathcal{C}\oplus\mathcal{C}^\perp$ and $\mathcal{C}^\perp$ is an
irreducible $kH$-module. Since $V$ is a non-trivial $kH$-module, we
obtain $V\cong \mathcal{C}^\perp$. Since $M$ is self dual, we obtain
that $M^*\cong M$ and hence $V^*\cong V$. Therefore, since $B$ fixes
the non-zero vector $b$ of $V$, it also fixes a
non-zero vector of $V^*$, but this contradicts Lemma~\ref{lemma1}. 

Now assume $p=3$.
From~\cite[Lemma~$2$]{Mortimer}, we have that $\mathcal{C}\subseteq
\mathcal{C}^\perp$ and $\mathcal{C}^\perp$ is the unique maximal
submodule of $M$. Therefore $W=\mathcal{C}^\perp$ and $V\cong
M/\mathcal{C}^\perp\cong\mathcal{C} $ is a trivial $kH$-module, a
contradiction. Therefore it 
remains to consider the case $p=2$.

Since $2$ divides $|\Omega|=24$, we have $\mathcal{C}\subseteq
\mathcal{C}^\perp$. From~\cite[Beispiele~$2$~$b)$]{Klemm}, we see
that $\mathcal{C}^\perp$ is the unique maximal submodule of $M$ and
hence $V\cong M/\mathcal{C}^\perp\cong\mathcal{C}$ is a trivial
$kH$-module, a contradiction.
\end{proof}

Now we are ready to prove Theorems~\ref{thm2},~\ref{thm:101} and Corollary~\ref{cornew}. 

\begin{proof}[Proof of Theorem~\ref{thm2}]We argue by
contradiction and we let $H$ be a minimal (with respect to the group
order) counterexample. Let $a,b\in V\setminus\{0\}$ with $|a^H|$
relatively prime to $|b^H|$.  

If $H$ is a cyclic group of prime order $p$, then every $H$-orbit on
$V\setminus\{0\}$ has size $p$, a contradiction. Similarly, from
Proposition~\ref{thmHsimple} we see that the group $H$ is not a non-abelian 
simple group. Thus $H$ has a 
non-identity proper normal subgroup $N$. From the Clifford correspondence,
$V_N=W_1\oplus \cdots \oplus W_k$ with $W_i$ a homogeneous $kN$-module,
for each $i\in \{1,\ldots,k\}$, and with $H$ acting transitively on
the set of direct summands $\{W_1,\ldots,W_k\}$ of $V$. (A
module is said to be \emph{homogeneous} if it is the direct sum of
pairwise isomorphic submodules.)  Write $a=\sum_{i=1}^ka_i$ and
$b=\sum_{i=1}^kb_i$ with $a_i,b_i\in W_i$, for each $i\in
\{1,\ldots,k\}$. Let $i$ and $j$ be in $\{1,\ldots,k\}$ with $a_i\neq
0$ and $b_j\neq 0$. Since $H$ acts transitively on
$\{W_1,\ldots,W_k\}$, there exist $h$ and $k$ in $H$ with $W_i^h=W_1$
and $W_j^k=W_1$. In particular, replacing $a$ and $b$ by $a^h$ and
$b^k$ if necessary, we may assume that $i=j=1$. Since $N$ is a normal
subgroup of $H$, we get that $|a^N|$  divides
$|a^H|$ (respectively $|b^N|$ divides $|b^H|$). Furthermore, since $N$ acts trivially on
$\{W_1,\ldots,W_k\}$, we obtain that $\cent N a\subseteq \cent N
{a_1}$ and $\cent N b\subseteq \cent N {b_1}$, that is, $|a_1^N|$
divides $|a^N|$ and $|b_1^N|$ divides $|b^N|$. In particular,
$|a_1^N|$ and $|b_1^N|$ are coprime. 

As $W_1$ is a homogeneous $kN$-module, there
exists an irreducible 
$kN$-module $U$ such that $W_1=U_1\oplus
\cdots\oplus U_r$ with $U_i\cong U$, for each $i\in
\{1,\ldots,r\}$. Assume that $U$ is the trivial $kN$-module. So, $N$ acts trivially on $W_1$. Since $H$ permutes
transitively the set of direct summands $\{W_1,\ldots,W_k\}$ and since $N$ is normal in $H$, we obtain that $N$ acts 
trivially on $V$ and $N=1$, a contradiction. Therefore $U$ is a non-trivial irreducible $kN$-module and, in particular, $|a_1^N|,|b_1^N|>1$.

Write $a_1=\sum_{i=1}^rx_i$ and
$b_1=\sum_{i=1}^ry_i$ with $x_i,y_i\in U_i$, for each $i\in
\{1,\ldots,r\}$. Since $N$ stabilizes the direct summands
$\{U_1,\ldots,U_r\}$ of $U$, we obtain that $\cent N
   {a_1}=\cap_{i=1}^r\cent N {x_i}$ and $\cent N
   {b_1}=\cap_{i=1}^r\cent N {y_i}$. In particular, for each $i\in
   \{1,\ldots,r\}$, we have that $|x_i^N|$ divides $|a_1^N|$ and
   $|y_i^N|$ divides $|b_1^N|$.
Since $a_1,b_1\neq 0$, there exist $i,j\in \{1,\ldots,r\}$ with
$x_i\neq 0$ and $y_j\neq 0$. Fix $\varphi_i:U_i\to U$ and
$\varphi_j:U_j\to U$ 
two $N$-isomorphisms and write $x=\varphi_i(x_i)$ and
$y=\varphi_j(y_j)$. In particular, $x$ and $y$ are non-zero elements
of the non-trivial irreducible $kN$-module $U$. Furthermore, since
$\varphi_i$ and 
$\varphi_j$ are $N$-isomorphisms, we obtain $\cent N {x_i}=\cent N x$
and $\cent N {y_j}=\cent N y$ and thus $|x^N|$ and $|y^N|$ are
coprime. This contradicts the minimality of $H$ and hence the theorem
is proved.
\end{proof}

\begin{proof}[Proof of Corollary~\ref{cornew}]
We argue by contradiction and we assume that $V$ is a non-trivial finite dimensional $kH$-module and that $a$ and $b$ are elements of $V$ with $V=\langle a^h\mid h\in H\rangle=\langle b^h\mid h\in H\rangle$ and with $\gcd(|a^H|,|b^H|)=1$. Now we argue by induction on $\dim_kV$. If $V$ is irreducible, then the result follows from Theorem~\ref{thm2}. So, we assume that this is not the case.  Let $W$ be a minimal submodule of $V$ and  suppose that $V/W$ is non-trivial. Clearly, $(a+W)^H$ and $(b+W)^H$ span $V/W$ and hence, by induction,  the lengths of the orbits of $(a + W)^H$ and $(b + W)^H$ are not coprime. As $|(a+W)^H|$ divides $|a^H|$ and $|(b+W)^H|$ divides $|b^H|$, we have that $|a^H|$ and $|b^H|$ are not coprime. 

Suppose now that $V/W$ is the trivial $kH$-module. We claim that in this case $V$ splits over $W$, that is, $V=\langle v\rangle\oplus W$ for some element $v$ of $V$ fixed by $H$. If the characteristic of $V$ is zero, then $V$ is semisimple and our claim is immediate. Suppose that $V$ has characteristic $p>0$. Replacing $a$ by $b$ if necessary, we may assume that $p\nmid |a^H|$ and hence $\cent H a$ contains a Sylow $p$-subgroup $P$ of $H$.
We claim that $V\cong k \oplus W$, that is, $V$ splits over $W$. The module $V$ corresponds to an element $\delta$ of $\mathrm{Ext}_G ^1(k,W)\cong H^1(G,W^*)$ (see~\cite[Section~(III)~$2$]{Brown} for the last isomorphism). On the other hand, $V$ splits over $W$ as a $kP$-module because $P\subseteq \cent H a$ and $a\notin W$. Thus $\delta=0$  in $H^1(P,W^*)$.
However, from~\cite[Theorem~$10.3$]{Brown}, we see that the restriction map $\textrm{res}:H^1(G,W^*) \rightarrow H^1(P,W^*)$ is injective. So $\delta=0$ is $H^1(G,W^*)$ and $V$ splits over $W$. In particular, $H$ fixes a vector $v\in V\setminus W$ and $V = \langle v\rangle \oplus W$. 

Write $a=\lambda v+a'$ and $b=\mu v+b'$ with $\lambda,\mu\in k$, $a'\in W$ and $b'\in W$. Clearly, $a',b'\neq 0$ because $a^H$ and $b^H$ span $V$ and $V$ is not the trivial module. Similarly, $W$ is not the trivial $kH$-module. Since $H$ fixes $v$, we have $\cent H {a}=\cent H {a'}$ and $\cent H b=\cent H {b'}$ and hence $|a'^H|,|b'^H|$ are relatively prime. This contradicts Theorem~\ref{thm2} applied to the irreducible module $W$ and to the vectors $a',b'$.  
\end{proof}

\begin{proof}[Proof of Theorem~\ref{thm:101}]
Write $A = \cent H a$.
Since $H$ is $p$-soluble and $A$ contains a Sylow $p$-subgroup of $H$, the group $H$ contains a $p'$-subgroup $L$ with $H=AL$. (For example, $H=AL$ for each Hall $p'$-subgroup $L$ of $H$.) Now, let $L$ be any $p'$-subgroup of $H$ with $H=AL$ and define $\psi_L:V\to V$ by setting 
$$\psi_L(v)=\sum_{x\in L}v^x.$$
We claim that $\cent V L=\psi_L(V)$. For $v\in V$ and $y\in L$, we have $$\psi_L(v)^y=\left(\sum_{x\in L}v^{x}\right)^y=\sum_{x\in L}v^{xy}=\sum_{x\in L}v^x=\psi_L(v).$$ So $\psi_L(V)\subseteq \cent V L$. Conversely, if $v\in \cent V L$, then 
$$\psi_L(v)=\sum_{x\in L}v^x=\sum_{x\in L}v=|L|v.$$ As $|L|$ is coprime to $p$, we have $v=\psi_L(v/|L|)\in \psi_L(V)$.

We now show that $\cent V H=\cent V L$. As $L\subseteq H$, we have $\cent V H\subseteq \cent V L$.  As $H=AL$, we have $a^{H}=a^{AL}=a^L$. So, for every $v\in a^H$, the image $\psi_L(v)$ is a multiple of the sum of the elements of $a^L=a^H$. We deduce that $\psi_L(v)$ is $H$-invariant, that is, $H$ fixes $\psi_L(v)$. Since $a^H$ spans $V$, we obtain that $H$ fixes every element of $\psi_L(V)=\cent V L$, that is, $\cent VL\subseteq \cent VH$.

Now we are ready to prove~$(a)$. Since $a^H=a^L$ and $a^H$ spans $V$, the vector space $V$ is generated by $a$ as a $kL$-module. Thus, the map $\pi:kL\to V$, given by $\pi(\sum_{x\in L}\alpha_xx)=\sum_{x\in L}\alpha_xa^x$, defines a $kL$-homomorphism of $kL$ onto $V$. Since $p$ is coprime to $|L|$, by Maschke's theorem the $kL$-module $V$ is isomorphic to a direct summand of the group-algebra $kL$. Therefore $\dim \cent V H=\dim \cent V L\leq \dim \cent {kL}{L}=1$.

We now prove~$(b)$. Let $b\in V$ with $\gcd(|a^H|,|b^H|)=1$. Write
$B=\cent V b$ and observe that $H=AB$. Since $a$ is fixed by a Sylow $p$-subgroup of $H$, we see that $p$ does not divide $|H : A| = |B :
(A \cap B)|$ and so $A \cap B$ contains a Sylow $p$-subgroup of $B$. As $H$ is $p$-soluble, we get
that $B$ is $p$-soluble and that $B$ contains a $p$-complement $L$, say. So, $B=(A\cap B)L$ and $H=AB=AL$. In particular, we are in the position to apply the first part of the proof to $L$. Thus $b\in \cent V B\subseteq \cent V L=\cent V H$. 
\end{proof}

\section{Proof of Theorem~\ref{thm1}}\label{secthm1}

In this section we use Theorem~\ref{thm2} to prove
Theorem~\ref{thm1}. We start by showing that the hypothesis
``completely reducible'' is essential.

\begin{example}\label{Exa1}{\rm 
Let $p$ be an odd prime, $V$ be the $2$-dimensional vector space of row vectors over a field $\mathbb{F}_p$ of size $p$, $\lambda$ be a generator of the multiplicative group $\mathbb{F}_p\setminus\{0\}$ and \[
H=\langle g,h\rangle\textrm{ with }g=
\left(
\begin{array}{cc}
1&0\\
1&1\\
\end{array}
\right), h=
\left(
\begin{array}{cc}
\lambda&0\\
0&1\\
\end{array}
\right).
\]
The group $H$ has order $p(p-1)$ and has $p+1$ orbits on $V$. Namely, for each $a\in \mathbb{F}_p\setminus\{0\}$, the set $\{(x,a)\mid x\in\mathbb{F}_p\}$ is an $H$-orbit of size $p$. Furthermore, $\{(0,0)\}$ and $\{(a,0)\mid a\in \mathbb{F}_p\setminus\{0\}\}$ are $H$-orbits of size $1$ and $p-1$, respectively. Write $e_1=(\lambda,0)$ and $e_2=(0,1-\lambda)$. We have $\cent H {e_1}=\langle g\rangle$, $\cent H{e_2}=\langle h\rangle$ and $\cent H{e_1+e_2}=\langle gh\rangle\neq \cent H{e_1}\cap \cent H {e_2}=1$. 

Here is an example for the prime $p=2$. Let $H=\Sym(4)$ be the symmetric group of
  degree $4$ and $M$ the permutation module with basis
  $e_1,e_2,e_3,e_4$ over a field $k$ of size $2$. It is easy to see that
  the only $kH$-submodules of $M$ are $0$, $M_1=\langle e_1+e_2+e_3+e_4\rangle$,
  $M_2=\langle e_1+e_2,e_1+e_3,e_1+e_4\rangle$ and $M$, and that $0\subset
  M_1\subset M_2\subset M$, that is, $M$ is uniserial. Let $V$ be the $kH$-module
  $M/M_1$. Clearly, $H$ acts faithfully on $V$ and, as $M_2/M_1$ is
  the unique proper submodule of $V$, we have that $V$ is not
  completely reducible. Write $a=e_1+M_1$ and
  $b=e_1+e_2+M_1$. We have $\cent H a=\langle (2,3),(3,4)\rangle$ and
  $\cent H b=\langle (1,2),(1,3,2,4)\rangle$ and so $a^H$ has size $4$
  and $b^H$ has size $3$. Finally, $\cent H {a+b}=\cent H
  {e_2+M_1}=\langle (1,3),(3,4)\rangle\neq \langle (3,4)\rangle=\cent
  H a\cap \cent H 
  b$. Furthermore, the orbits of $H$ on $V$ have sizes $1,3$ and $4$.} 
\end{example}

We note that an example similar to Example~\ref{Exa1} is
in~\cite[Example~$1$]{Isaacs3}.

\begin{proof}[Proof of Theorem~\ref{thm1}]
As $V$ is completely
reducible, we have $V=\cent V H \oplus W$ for some direct summand $W$
of $V$. Clearly, replacing $V$ by $W$ if necessary, we may assume that
$\cent V H=0$, that is, $H$ fixes no non-zero vector of $V$. 
There is also no loss in assuming that $V$ is generated by 
$a$ and $b$ as a $kH$-module. Let $V(a)$ and $V(b)$ denote the $kH$-submodules generated by $a$ and $b$, respectively. We claim that $V(a) 
\cap V(b) = 0$, whence  $V=V(a)\oplus V(b)$ and $\cent H {a+b} = \cent H a \cap \cent Hb$
as required.   

Suppose not.  Let $S$ be a simple $kH$-submodule of
$V(a)\cap V(b)$. Since $\cent H V=0$, $H$ does not act trivially on $S$.
Since $V$ is completely reducible, $V = S \oplus T$ as $kH$-modules.
Let $\pi$ denote the projection of $V$ onto $S$ with kernel $T$.
Clearly, $|a^H|$ is a multiple of $|\pi(a)^H|$ and similarly for $b$.
Since $S \le V(a)$, $\pi(a) \ne 0$ (and similarly for $b$).
Thus, the lengths of the $H$-orbits in $S$ of  $\pi(a)$ and $\pi(b)$
are coprime contradicting Theorem~\ref{thm2}.  
\end{proof}

We point out that from Theorem~\ref{thm1} we can easily deduce the
following well-known result of Yuster (see~\cite{Yuster} or
\cite[3.34]{Isaacs}). 

\begin{corollary}
Let $H$ and $A$ be finite groups with $|H|$ relatively prime to $|A|$
and with $H$ acting as a group of automorphisms on $A$. If, for $a,b\in A$, the
$H$-orbits $a^H$ and $b^H$ have relatively prime
size, then $H$ has an orbit of size $|a^H||b^H|$. 
\end{corollary}

\begin{proof}
As $|H|$ is relatively prime to $|A|$,
from~\cite[Lemma~$2.6.2$]{HarTur} we see that we may assume that $A$
is a direct product of elementary abelian groups. In particular, from
Maschke's theorem, $A$ is a completely reducible $H$-module, possibly
of mixed characteristic. Now the
result follows from Theorem~\ref{thm1}.
\end{proof}

\section{Proof of Theorem~\ref{thm3}}\label{secthm3}

The main ingredient in the proof of Theorem~\ref{thm3} is
Theorem~\ref{thm2} and the positive solution of
Fisman and Arad~\cite{FZ} of Szep's conjecture. 

\begin{theorem}\label{szep}{\rm\cite{FZ}}
Let $G=AB$ be a finite group such that $A$ and $B$ are both
subgroups of $G$ with non-trivial centres. Then $G$ is not a
non-abelian simple group.
\end{theorem}

We start by considering some examples.

\begin{example}\label{exa2}{\rm \textsc{Primitive groups of AS type. }From~\cite{Liv}, we see that the
    sporadic simple group $G=J_1$ has a 
    primitive permutation representation of rank $5$ on a set $\Delta$
    of size $266$. The subdegrees of $G$ are $1,11,12,110$ and
    $132$. In particular, $G$ has two coprime subdegrees. No primitive
    group of smaller rank has this property: the proof of this
    assertion requires the
    classification of the
    finite simple groups~\cite[Remark, p.~33]{PJC}.

Now we give an infinite family of examples. Let $p$ be a prime with
$p\equiv \pm 1\mod 5$ and with $p\equiv \pm 1\mod 
16$, and let
$G=\PSL(2,p)$. From~\cite[Chapter~3, Section~~6]{Suzuki}, we see that $G$ contains a maximal
subgroup $H$ with $H\cong \Alt(5)$. Consider $G$ as a primitive
permutation group
acting on the set $\Delta$ of right cosets of $H$ in $G$. Let $K$ be a
maximal subgroup of $H$ with 
$K\cong \Alt(4)$. As $8$ divides $|G|$, we see from~\cite[Chapter~3,
  Section~6]{Suzuki} that 
$N_G(K)\cong \Sym(4)$. Let $g\in N_G(K)\setminus H$. Then $K=H\cap
H^g$, $|H:H\cap H^g|=5$ and so $G$ has a suborbit of size $5$ on
$\Delta$. Similarly, let now $K$ be a Sylow $5$-subgroup of $H$. Using the
generators of $H$ given in~\cite[Chapter~3, Section~6]{Suzuki},
we see, with a direct computation, that there exists $g\in N_G(K)\setminus
H$ with $K=H\cap H^g$. Therefore $|H:H\cap H^g|=12$ and so $G$ has a
suborbit of size $12$. Furthermore, another explicit computation with
  the generators of $H$ shows that there exists $g\in G$ with $H\cap
  H^g=1$. So $G$ has a suborbit of size $60=5\cdot 12$.}
\end{example}
\begin{example}\label{exa3}{\rm \textsc{Primitive groups of PA
      type. }Let $G$ be a primitive group of AS type on $\Delta$ with non-trivial coprime subdegrees $a$ and $b$. Let
    $\delta,\delta_1$ and $\delta_2$ be in $\Delta$ with
    $a=|\delta_1^{G_\delta}|$, $b=|\delta_2^{G_\delta}|$.
For each $n\geq 2$, the wreath product $W=G\wr\Sym(n)$ endowed with
its natural product action on $\Omega=\Delta^n$ is a primitive group
of PA type. Consider the elements
$\alpha=(\delta,\ldots,\delta)$,
$\beta=(\delta_1,\ldots,\delta_1)$ and
$\gamma=(\delta_2,\ldots,\delta_2)$ of $\Omega$. We have
$|\beta^{W_\alpha}|=a^n$ and $|\gamma^{W_\alpha}|=b^n$ and so $a^n$ and
$b^n$ are two coprime subdegrees of $W$. In many cases there are several pairs of coprime non-trivial subdegrees of $W$. For example, if $n\geq 3$ and $n$ is coprime to $b$, then the point $\beta'= (\delta_1,\delta,\ldots,\delta)$ lies in a $W_\alpha$-orbit of size $na$ and we have also $na$ and $b^n$ as coprime non-trivial subdegrees.

In particular, this construction can be applied with $G$ and $\Delta$ as in
Example~\ref{exa2}.} 
\end{example}

\begin{example}\label{exa4}{\rm \textsc{Primitive groups of TW type. }
In this example we construct a primitive group of TW
type with two non-trivial coprime subdegrees. We start by recalling the structure
and the action of a primitive group of twisted wreath type. We
follow~\cite[Section~$4.7$]{DM}. Let $T$ be a non-abelian simple
group, $H$ be a group, $L$ be a subgroup of $H$ and $\varphi:L\to\Aut(T)$
be a homomorphism with the image of $\varphi$ containing the inner
automorphisms 
of $T$. Let $R$ be a set of left coset representatives of $L$ in $H$
and $T^H$ be the set of all functions $f:H\to T$ from $H$ to
$T$. Clearly, $T^H$ is a 
group under pointwise multiplication, and $H$ acts as a group of
automorphisms on $T^H$ by setting $f^x(z)=f(xz)$, for $f\in T^H$ and for
$x,z\in H$. Write $N=\{f\in T^H\mid f(zl)=f(z)^{\varphi(l)}\textrm{ for all
}z\in H \textrm{ and }l\in L \}$. It is easy to verify that $N$ is an
$H$-invariant  subgroup of $T^H$ isomorphic to $T^{R}$. In fact, the
restriction mapping $f\mapsto f\mid_{R}$ is an isomorphism of $N$ onto
$T^R$. The semidirect product $G=N\rtimes H$ is said to be the \emph{twisted
wreath product} determined by $H$ and $\varphi$. The group $G$ acts on
$\Omega=N$ by setting 
$\alpha^{nh}=(\alpha n)^h$, for each $\alpha\in \Omega$, $n\in N$ and
$h\in H$. (In particular, $N$ acts on $\Omega$ by right multiplication
and $H$ acts on $\Omega$ by conjugation.) From~\cite[Lemma~4.7B]{DM},
we see that if $H$ is a 
primitive permutation group with point stabilizer $L$ and if the image of
$\varphi$ is not a homomorphic image of $H$, then $G$ acts primitively
on $\Omega$ and the socle of $G$ is $N$.

Write $H=\PSL(2,7)^2\rtimes\langle\iota\rangle$ where
$\iota$ is the involutory automorphism of
$\PSL(2,7)^2=\PSL(2,7)\times\PSL(2,7)$ defined by 
$(x,y)^\iota=(y,x)$. Write $L=\{(x,x)\mid
x\in\PSL(2,7)\}\langle\iota\rangle$. The group $H$ is a primitive group
of SD type in its  action on the right cosets of $L$. Let
$T$ be $\PSL(2,7)$ and let $\varphi:L\to \Aut(T)$ be the mapping sending
$(x,x)\iota^i$ to the 
inner automorphism of $T$ given by $x$, for each $x\in \PSL(2,7)$ and
$i\in\{0,1\}$. Clearly, $\varphi$ is a 
homomorphism whose image contains the inner automorphisms of
$T$. Let $G$ be the twisted wreath product determined by $H$ and
$\varphi$ and let $N$ be its socle as described above. Since $H$ has
no homomorphic image  
isomorphic to $T$, we see that $G$ is a primitive permutation group on
$\Omega=N$. The group $H$ is the stabilizer $G_{f}$ of the function
$f:H\to T$ mapping every element of $H$ to the identity of $T$.

Write
\[
\gamma=\left[
\begin{array}{cc}
1&1\\
0&1\\
\end{array}
\right]
\]
(regarded as an element of $T$). Let $C=(\langle \gamma\rangle\times
\langle\gamma\rangle)\rtimes\langle\iota\rangle$ as a subgroup of $H$
and define $g:H\to T$ by 
\[
g(z)=\left\{
\begin{array}{ccl}
\gamma^{\varphi(l)}&&\textrm{if }z=cl,\textrm{ for some }c\in C
\textrm{ and } l\in L,\\ 
1 &&\textrm{if }z\in H\setminus CL.
\end{array}
\right.
\]
We claim that the function $g$ is well-defined. In fact, if $z=c_1l_1=c_2l_2$ for
$c_1,c_2\in C$ and $l_1,l_2\in L$, then $l_2l_1^{-1}\in C\cap
L=\langle(\gamma,\gamma)\rangle\langle\iota\rangle$. 
Hence $l_2=ul_1$ with $u=(\gamma^k,\gamma^k)\iota^i$ for some
$k\in \{0,\ldots,6\}$ and $i\in \{0,1\}$. In particular, since
$\gamma^k$ centralizes $\gamma$, we obtain
$\gamma^{\varphi(l_1)}=\gamma^{\varphi(u)\varphi(l_1)}
=\gamma^{\varphi(ul_1)}=\gamma^{\varphi(l_2)}$ and hence the image
$\gamma^{\varphi(l)}$ is independent of the representation $z=c_il_i$
of $z$. 

Fix $z$ in $H$. Distinguishing
the cases $z\in CL$ and $z\notin CL$, it is easy to verify that
$g(zl)=g(z)^{\varphi(l)}$ for each $l\in L$ and $z\in H$, and hence $g\in \Omega$. For each $c\in C$ and
$z\in H$, we
have $g^c(z)=g(cz)=g(z)$, and hence $C\subseteq \cent H
g$. We claim that $C=\cent H g$. Let $h=(h_1,h_2)\iota^j $ be in $\cent
H g$. Suppose that  $h\notin CL$. As $g(1)=\gamma\neq
1$ and $g^h(1)=g(h)=1$, we obtain $g\neq g^h$. Thus $h\in CL$ and
$\cent H g\subseteq CL$. As $C\subseteq \cent H g$, replacing $h$ by
$ch$ for a suitable element $c\in C$, we may assume that
$h\in L$, that is, $h=(x,x)\iota^i$ for some $x\in
T$ and $i\in \{0,1\}$. Now
$\gamma=g(1)=g^h(1)=g(h)=\gamma^{\varphi(x)}$. Hence $x\in
\cent T \gamma=\langle\gamma\rangle$, $h\in C\cap L$ and our claim is proved. 
Thus the
$H$-orbit containing $g$ has size $|H:C|=24^2=576$, and $576$ is a
subdegree of $G$.

Write
\[
a=\left[
\begin{array}{cc}
0&4\\
5&4\\
\end{array}
\right]\quad\textrm{ and }\quad
b=\left[
\begin{array}{cc}
2&1\\
0&4\\
\end{array}
\right]
\] (again thought of as elements of $T$). The element $a$ has order
$4$, the element $b$ has order $3$, and 
$\langle a ,b\rangle\cong \Sym(4)$.  
Let $D=(\langle a,b\rangle\times \langle
a,b\rangle)\rtimes\langle\iota\rangle$ and let $t=(\gamma,1)$. A direct
computation shows 
that $D^t\cap L$ is a dihedral group of size $8$, namely
\[
\langle (a^2,a^2),(r,r)\iota\rangle\quad\textrm{ with }
r=\left[
\begin{array}{cc}
3&5\\
4&0\\
\end{array}
\right]
\]
and with centre
\[
\langle (\eta,\eta)\rangle 
\quad\textrm{ where } \eta=r^2=
\left[
\begin{array}{cc}
1&1\\
5&6\\
\end{array}
\right].
\] 
Define
\[
h(z)=\left\{
\begin{array}{ccl}
\eta^{\varphi(l)}&&\textrm{if }z=dtl,\textrm{ for some }d\in D
\textrm{ and } l\in L,\\
1 &&\textrm{if }z\notin DtL.
\end{array}
\right.
\]
We claim that the function $h$ is well-defined. In fact, if
$z=d_1tl_1=d_2tl_2$ for 
$d_1,d_2\in D$ and $l_1,l_2\in L$, then
$l_2l_1^{-1}=t^{-1}d_2^{-1}d_1t\in D^t\cap L$. 
Hence $l_2=ul_1$ with $u\in D^t\cap L$. In particular, since
$u$ centralizes $(\eta,\eta)$, we obtain
$\eta^{\varphi(l_1)}=\eta^{\varphi(u)\varphi(l_1)}=
\eta^{\varphi(ul_1)}=\eta^{\varphi(l_2)}$ and so the image
$\eta^{\varphi(l)}$ is independent of the representation $z=d_itl_i$ of $z$.  

Fix $z$ in $H$. As above, by distinguishing the cases $z\in DtL$ and
$z\notin DtL$, it is easy to verify that $h(zl)=h(z)^{\varphi(l)}$ for each
 $l\in L$, and hence $h\in \Omega$. From the definition of $h$, we see
 that for each $d\in D$ and
$z\in H$, we
have $h^d(z)=h(dz)=h(z)$. Hence $D\subseteq \cent H
h$. Since $D$ is a maximal subgroup of $H$, we obtain $D=\cent H h$. Thus the
$H$-orbit containing $h$ has size $|H:D|=7^2=49$, and $49$ is a
subdegree of $G$. 

This shows that $G$ has two coprime subdegrees, namely $576$ and $49$. Now, using the computational algebra system \texttt{GAP}~\cite{GAP4}, it is easy to show that $\cent H {fg}=\cent H f\cap \cent H g$. In particular, the suborbit of $G$ containing $fg$ has size $576\cdot 49$.} 
\end{example}

\begin{proof}[Proof of Theorem~\ref{thm3}]From
Examples~\ref{exa2},~\ref{exa3} and~\ref{exa4} and Theorem~\ref{thm2},
it suffices to show that if 
$G$ is a primitive group of type HS, HC, SD or CD, then $G$ has no non-trivial
coprime
subdegrees.  We argue by contradiction and we assume that $G$ is a
primitive group on $\Omega$ of HS, HC, SD or CD type with two
non-trivial coprime
subdegrees. 

We first consider the case that $G$ is of HS or HC type. Let
$N=T_1\times \cdots\times T_\ell$ be the socle of $G$, with $T_i\cong
T$ for some  non-abelian simple
group $T$, for each $i\in \{1,\ldots,\ell\}$. From the description of the 
O'Nan-Scott types in~\cite{Pr}, we see that $\ell$ is
even (with $\ell=2$ if $G$ is of HS type and with $\ell>2$ if $G$ is
of HC type). Furthermore, 
relabelling the indices $\{1,\ldots,\ell\}$ if necessary,
$M_1=T_1\times \cdots \times T_{\ell/2}\cong T^{\ell/2}$ and
$M_2=T_{\ell/2+1}\times\cdots \times T_{\ell}=T^{\ell/2}$ are minimal normal
regular subgroups of $G$, and $\Omega$ can be identified with
$T^{\ell/2}$. Namely, the 
action of $N$ on $\Omega$ is permutation isomorphic to the action of
$T^{\ell/2}\times T^{\ell/2}$ on $T^{\ell/2}$ given
by $$z^{(a,b)}=a^{-1}zb=(a_1^{-1}z_1b_1,\ldots,a_{\ell/2}^{-1}z_{\ell/2}b_{\ell/2}),$$
for
$a=(a_1,\ldots,a_{\ell/2}),b=(b_1,\ldots,b_{\ell/2}),z=(z_1,\ldots,z_{\ell/2})\in
T^{\ell/2}$. Under this identification, the stabilizer in
$T^{\ell/2}\times T^{\ell/2}$ of the
element $(1,\ldots,1)$ of 
$T^{\ell/2}$ is the set $\{(a,a)\mid a\in T^{\ell/2}\}$ acting on $T^{\ell/2}$ by
conjugation, that is, $z^{(a,a)}=a^{-1}za$.  

Let
$\omega_1,\omega_2$ be elements of 
$\Omega\setminus\{\omega\} $ with $m=|\omega_1^{G_{\omega}}|$ coprime to
$n=|\omega_2^{G_\omega}|$.  Since $N$ is
normal in $G$, we have that $m'=|\omega_1^{N_\omega}|=|N_\omega:N_{\omega,\omega_1}|$
divides $m$ and $n'=|\omega_2^{N_\omega}|=|N_\omega:{N_{\omega,\omega_2}}|$ divides
$n$. In particular, $m'$ and $n'$ are coprime. Identifying $\Omega$
with $T^{\ell/2}$ (as above), $\omega$ with $(1,\ldots,1)$, $N_\omega$
with $T^{\ell/2}$ (as above), $\omega_1$ with
$(x_1,\ldots,x_{\ell/2})$ and $\omega_2$ with
$(y_1,\ldots,y_{\ell/2})$, we get
\begin{eqnarray*}
N_{\omega,\omega_1}&\cong &\cent
{T^{\ell/2}}{(x_1,\ldots,x_{\ell/2})}=\cent T {x_1}\times\cdots\times
\cent T {x_{\ell/2}}\\
N_{\omega,\omega_2}&\cong &\cent
{T^{\ell/2}}{(y_1,\ldots,y_{\ell/2})}=\cent T {y_1}\times\cdots\times
\cent T {y_{\ell/2}}.\\
\end{eqnarray*}
In particular, $m'=\prod_{i=1}^{\ell/2}|T:\cent T {x_i}|$ and
$n'=\prod_{i=1}^{\ell/2}|T:\cent T {y_i}|$. As $m'$ is coprime to
$n'$, for each $i,j\in \{1,\ldots,\ell/2\}$, the indices $|T:\cent
T{x_i}|$ and $|T:\cent T{y_j}|$ are coprime and hence $T=\cent T
{x_i}\cent T {y_j}$.   As $\omega_1,\omega_2\neq 1$, 
there exist $i,j\in \{1,\ldots,\ell/2\}$ with $x_i\neq 1$ and
$y_j\neq 1$. So $T=\cent T {x_i}\cent T {y_j}$ is a coprime
factorization and
Theorem~\ref{szep} yields  that $T$ is not a non-abelian simple group, a
contradiction.

It remains to consider the case that G is of SD or CD type. From the
description of the O'Nan-Scott types in~\cite{Pr}, we  may write the
socle $N$ of $G$ as $$
N=
(T_{1,1}\times \cdots\times T_{1,r}) \times
(T_{2,1}\times \cdots \times T_{2,r})\times\cdots \times
(T_{s,1}\times \cdots\times  T_{s,r}),$$
 with
$T_{i,j}\cong T$ for some non-abelian simple group $T$, for each $i\in
\{1,\ldots,r\}$ and $j\in \{1,\ldots,s\}$: where $r\geq 2$ and $s\geq 1$
(with $s=1$ if $G$ is of SD type and with $s\geq 2$ if $G$ is of CD
type).  The set $\Omega$
can be identified with $
(T_{1,1}\times\cdots \times T_{1,r-1})\times
(T_{2,1}\times\cdots \times T_{2,r-1})\times\cdots \times
(T_{s,1}\times\cdots\times T_{s,r-1})\cong T^{s(r-1)}$ and, for the point
$\omega\in \Omega$ identified with $(1,\ldots,1)$, the stabilizer 
$N_\omega$ is $D_1\times \cdots \times D_s\cong T^s$ where $D_i$ is the
diagonal subgroup $\{(t,\ldots,t)\mid t \in T\}$ of $T_{i,1}\times
\cdots\times T_{i,r}$. That is to say, the action
of $N_\omega$ on $\Omega$ is permutation isomorphic to the action of
$T^s$ on $T^{s(r-1)}$ by ``diagonal'' componentwise conjugation: the
image of the point
$(x_{1,1},\ldots,x_{1,r-1},\ldots,x_{s,1},\ldots,x_{s,r-1})$  under the
permutation $(t_{1},\ldots,t_{s})$  is

$$(x_{1,1}^{t_{1}},\ldots,x_{1,r-1}^{t_1},x_{2,1}^{t_2},\ldots,x_{2,r-1}^{t_2},\ldots,
x_{s,1}^{t_{s}},\ldots,x_{s,r-1}^{t_s}).$$
Let $\omega_1,\omega_2$ be elements of $\Omega\setminus\{\omega\}$ with
$m=|\omega_1^{G_\omega}|$ coprime to $n=|\omega_2^{G_\omega}|$. Since $N$
is normal in $G$, we have that
$m'=|\omega_1^{N_\omega}|=|N_\omega:N_{\omega,\omega_1}|$ divides $m$ and
$n'=|\omega_2^{N_\omega}|=|N_\omega:N_{\omega,\omega_2}|$ divides $n$. In
particular, $m'$  and $n'$ are coprime. Identifying  $\omega_1$ with
$(x_{1,1},\ldots,x_{1,r-1},\ldots,x_{s,1},\ldots,x_{s,r-1})$
 and $\omega_2$ with
$(y_{1,1},\ldots,y_{1,r-1},\ldots,y_{s,1},\ldots,y_{s,r-1})$, we get
\begin{eqnarray*}
N_{\omega,\omega_1}&\cong&\cent {T}{\langle x_{1,1},\ldots,x_{1,r-1}\rangle}
\times \cdots \times
\cent {T}{\langle x_{s,1},\ldots,x_{s,r-1}\rangle}\\
N_{\omega,\omega_2}&\cong&\cent {T}{\langle
y_{1,1},\ldots,y_{1,r-1}\rangle}
\times \cdots \times
\cent {T}{\langle y_{s,1},\ldots,y_{s,r-1}\rangle}.\\
\end{eqnarray*}
In particular, $m'=\prod_{i=1}^s|T:\cent {T}{\langle
 x_{i,1},\ldots,x_{i,r-1}\rangle}|$
and
$n'=\prod_{i=1}^s|T:\cent {T}{\langle
y_{i,1},\ldots,y_{i,r-1}\rangle}|$.

As $\omega_1,\omega_2\neq \omega$, there exist $i_1,i_2\in \{1,\ldots,r-1\}$
and $j_1,j_2\in \{1,\ldots,s\}$ with $x_{i_1,j_1}\neq 1$ and
$y_{i_2,j_2}\neq 1$. Now, since  $\cent {T}{\langle
x_{i_1,1},\ldots,x_{i_1,r-1}\rangle}\subseteq \cent
{T}{x_{i_1,j_1}}\subsetneq T$ and   $\cent {T}{
\langle x_{i_2,1},\ldots,x_{i_2,r-1}\rangle}\subseteq \cent
{T}{x_{i_2,j_2}}\subsetneq T$ and since $m'$ is relatively prime to $n'$, we
obtain
$T=\cent {T}{x_{i_1,j_1}}\cent {T}{x_{i_2,j_2}}$. From Theorem~\ref{szep},
$T$ is a not non-abelian simple group, a contradiction.

\end{proof}

\section{Proofs of Theorems~\ref{TW} and~\ref{thm4}}\label{sec4}

Before proving Theorem~\ref{TW} and~\ref{thm4} we need the following definition and lemmas.
\begin{definition}{\rm If $G$ is a finite group, we let $\mu(G)$ denote the maximal size of a set $\{G_i\}_{i\in I}$ of proper subgroups of $G$ with $|G:G_i|$ and $|G:G_j|$ relatively prime, for each two distinct elements $i$ and $j$ of $I$.}
\end{definition}

\begin{lemma}\label{id1}
If $K$ is a direct product of isomorphic nonabelian simple groups, then $\mu(K)\leq 2$.
\end{lemma}
\begin{proof}
We have $K=T_1\times \cdots \times T_\ell$ with $T_i\cong T$, for some
nonabelian simple group $T$ and  for some $\ell\geq 1$. We argue by
contradiction and we assume that $\mu(K)\geq 3$, that is, $K$ has
three proper subgroups $A_1$, $A_2$ and $A_3$ with $|K:A_1|$,
$|K:A_2|$ and $|K:A_3|$ relatively prime. Write $a_i=|K:A_i|$ for
$i\in \{1,2,3\}$. Replacing $A_i$ with a maximal subgroup of $K$
containing $A_i$ if necessary, we may assume that $A_i$ is maximal in
$K$, for $i\in \{1,2,3\}$. In particular, (up to relabelling the
simple direct summands of $K$) we have that  either $A_1=M_1\times
T_2\times\cdots\times T_\ell$ (for some maximal subgroup $M_1$ of
$T_1$) or  $A_1=\{(t_1,t_1^\eta)\mid t_1\in T_1\}\times T_3\times
\cdots \times T_\ell$  (for some isomorphism $\eta:T_1\to T_2$). In
the latter case we have that $a_1=|T|$ is not relatively prime to
$a_2$ and to $a_3$, a contradiction. Therefore, up to relabeling the
indices, we may assume that  $A_i=M_i\times T_2\times \cdots\times
T_\ell$ with $M_i$ a maximal subgroup of $T_1$, for $i\in
\{1,2,3\}$. This gives that the nonabelian simple group $T_1$ admits
three coprime factorizations $T_1=M_1M_2=M_1M_3=M_2M_3$ with
$|T_1:M_1|,|T_1:M_2|$ and $|T_1:M_3|$ relatively prime. A simple
inspection in Table~\ref{table} shows that this is impossible. The
same conclusion can be obtained using~\cite{BadPr}, where the authors
determine all multiple factorizations of finite nonabelian simple
groups $T=M_{i}M_j$, for $i$ and $j$ distinct elements of
$\{1,2,3\}$. In particular, it is readily checked that in none of the
multiple factorizations in~\cite{BadPr} are the indices $|T:M_1|$,
$|T:M_2|$ and $|T:M_3|$ pairwise coprime. 
\end{proof}

\begin{lemma}\label{id2}
Let $G$ be a transitive permutation group on $\Omega$ and let $\omega$ be in $\Omega$. Suppose
that $N$ is normal in $G_\omega$ and $N$ fixes a unique point on $\Omega$. Then the number of coprime subdegrees of $G$ is at most $\mu(N)$.
\end{lemma}

\begin{proof}
Let $O_1,\ldots,O_r$ be orbits of $G_\omega$ on $\Omega\setminus\{\omega\}$ of pairwise coprime sizes and let $\omega_i\in O_i$, for $i\in \{1,\ldots,r\}$. Now the orbits of $N$ on $O_i$ have all the same size, $m_i$ say, and $m_i$ divides $|O_i|$. Since $N$ fixes only the point $\omega$ of $\Omega$, we have that $m_i>1$. Therefore $\{N_{\omega_i}\}_{i\in\{1,\ldots,r\}}$ is a set of proper subgroups of $N$ with pairwise coprime indices. Thus $r\leq \mu(N)$.
\end{proof}

\smallskip

\begin{proof}[Proof of Theorem~\ref{TW}]Let $G$ be a primitive group
  of TW type. We use the notation introduced in the first paragraph of
  Example~\ref{exa4}: so $G=N\rtimes H$ is the twisted wreath product
  determined by $H$ and $\varphi:L\to\Aut(T)$. Recall that $G$ acts
  primitively on $N$, with $N$ acting on itself by right
  multiplication and with $H$ acting on $N$ by conjugation. In
  particular, $H$ is the stabilizer of the point $1\in N$.  Let $K$ be
  a minimal normal subgroup of
  $H$. From~\cite[Theorem~4.7B~$(i)$]{DM}, $K$ is a direct product of
  nonabelian simple groups. Write $\ell=|H:L|$. Hence
  $N=T_1\times\cdots\times T_\ell$ with $T_i\cong T$, for each $i\in
  \{1,\ldots,\ell\}$. Furthermore, $L=N_H(T_i)$ for some $i\in
  \{1,\ldots,\ell\}$. Relabeling the $T_j$ if necessary, we may assume
  that $i=1$. From~\cite[Theorem~4.7B~$(ii)$]{DM}, the group $L$ is a
  core-free subgroup of $H$ and hence $K\nsubseteq L$.    

We claim that $K\cap L$ acting by conjugation on the simple group
$T_1$ induces all the inner automorphisms. If not, then $K\cap L\subseteq
\cent H {T_1}$ because $K\cap L$ is a normal subgroup of $L$ and $T_1$
is nonabelian simple. Thus the homomorphism $\varphi:L\to \Aut(T)$ can
be extended to a homomorphism $\hat{\varphi}:KL\to \Aut(T)$ of the
group $KL$ by setting $\hat{\varphi}(kl)=\varphi(l)$, for each $l\in
L$ and $k\in K$. As $L\subsetneq KL$, this contradicts the maximality condition
of $H$ in $G$ given in~\cite[Lemma~$3.1$~$(b)$]{Bad}, and the claim is
proved. In particular, since $K$ is a normal subgroup of $H$ and since
$H$ acts transitively on the $\ell$ simple direct summands
$\{T_1,\ldots,T_\ell\}$, we obtain that $K\cap N_H(T_i)$ induces by
conjugation all the inner automorphisms of $T_i$, for each $i\in
\{1,\ldots,\ell\}$. This gives $\cent N{K}=1$ and so $K$ fixes a
unique point of $N$. Now the proof follows from Lemmas~\ref{id1}
and~\ref{id2}.
\end{proof}

Before concluding this section we show that coprime
subdegrees in primitive groups $G$ of AS or PA type determine coprime
subdegrees  
 in transitive non-abelian simple groups $T$, and we give a strong
relation between $G$ and $T$. Let $G$ be a primitive group
of AS or PA type. We recall that from the
description of the O'Nan-Scott types in~\cite{Pr} the group $G$ is a
subgroup of the wreath product $H\wr \Sym(\ell)$ endowed with
its natural product action on $\Delta^\ell$, with $H$ an almost simple primitive
group on $\Delta$ (we have $\ell=1$ and $G=H$ if $G$ is of AS type,
and $\ell>1$ if $G$ is of PA type). Furthermore, if $T$ is the socle
of $H$, 
then $N=T_1\times \cdots\times T_\ell$ is the socle of $G$, where
$T_i\cong T$ for each $i\in \{1,\ldots,\ell\}$. Write $G_i=N_G(T_i)$
and denote by $\pi_i:G_i\to H$ the natural projection. From~\cite{Pr}, we
see that replacing $H$ by $\pi_i(G_i)$ if necessary, we may assume
that $\pi_i(G_i)$ is surjective. In this case, we say that $H$ is the
{\em component subgroup} of $G$. In particular, if $G$ is of AS type, the component subgroup of $G$ is $G$ itself.

(We recall that the definition of pseudo-maximal subgroup is in Definition~\ref{def}.)
\begin{proposition}\label{propo}Let $G$ be a primitive permutation
  group of AS or PA type acting on $\Delta^\ell$ with component subgroup $H\subseteq  \Sym(\Delta)$ and let $T$ be the socle of $H$. For
  $\delta\in\Delta$, the stabilizer $T_\delta$ is a pseudo-maximal
  subgroup of $T$. Furthermore, if $G$ has $k$ non-trivial
 coprime subdegrees, then $T$ acting on $\Delta$ has 
at least $k$  non-trivial coprime  subdegrees.
\end{proposition}

\begin{proof}
As $H$ is a primitive group of AS type on $\Delta$, we have that
$H_\delta$ is a maximal subgroup of the almost simple group $H$ with
$T\nsubseteq H_\delta$, for each 
$\delta\in \Delta$. Therefore $T_\delta=T\cap H_\delta$  is a
pseudo-maximal subgroup of $T$ and, in particular, $N_H(T_\delta)=H_\delta$. Let
$\Lambda$ be the set of fixed points of $T_\delta$ on $\Delta$ and let
$\delta_1,\delta_2\in \Lambda$. By transitivity, there exists $h\in H$
with $\delta_1^h=\delta_2$, that is,
$T_\delta=T_{\delta_2}=T_{\delta_1^h}=T_{\delta_1}^h=T_{\delta}^h$. Therefore $h\in N_H(T_\delta)=H_\delta$ and $\Lambda$ is the
$H_\delta$-orbit containing $\delta$, that is, $\Lambda=\{\delta\}$
and $T_\delta$ fixes a unique point of $\Delta$. 

Let $\delta\in \Delta$ and write 
 $\alpha=(\delta,\ldots,\delta)\in \Delta^\ell$. Let
$N=T_1\times \cdots \times T_\ell$ be the socle of $G$. Clearly,
$G_\alpha\subseteq H_\delta\wr\Sym(\ell)$ and, as $G_\alpha$ is a
maximal subgroup of $G$ and as $N\subseteq G$, we obtain
$N_\alpha=(T_{1})_\delta\times \cdots \times(T_\ell)_\delta$. 

Assume that $G$ has $k$ non-trivial coprime subdegrees $n_1,\ldots,n_k$. Now, there exist
$\beta_i=(\delta_{i,1},\ldots,\delta_{i,\ell})$ 
with $n_i=|\beta_i^{G_\alpha}|$, for $i\in \{1,\ldots,k\}$.  Since $N$ is a
normal subgroup of $G$, we obtain that $n_i'=|\beta_i^{N_\alpha}|$ divides
$n_i$, and so $n_1',\ldots,n_k'$ are pairwise coprime. Furthermore 
\begin{eqnarray*}
\beta_i^{N_\alpha}&=&(\delta_{i,1},\ldots,\delta_{i,\ell})^{(T_1)_\delta\times
  \cdots \times (T_\ell)_\delta}=\delta_{i,1}^{(T_1)_\delta}\times
\cdots\times \delta_{i,\ell}^{(T_\ell)_\delta}\\
\end{eqnarray*}
for each $i\in \{1,\ldots,k\}$, and so $n_i'=\prod_{j=1}^\ell|\delta_{i,j}^{T_\delta}|$.

As $n_i'$ is coprime with $n_j'$, for each distinct $i$ and $j$ in $\{1,\ldots,k\}$, the subdegrees $|\delta_{i,x}^{T_\delta}|$ and $|\delta_{j,y}'^{T_\delta}|$ of $T$
acting on $\Delta$ are
coprime, for each $x,y\in \{1,\ldots,\ell\}$. Since for each $i\in \{1,\ldots,k\}$ we have $\beta_i\neq \alpha$, there exists $j_i\in
\{1,\ldots,\ell\}$ with $\delta_{i,j_i}\neq \delta$. Since $T_\delta$ fixes only the element $\delta$ of $\Delta$,
we have $|\delta_{i,j_i}^{T_\delta}|>1$. Thus $|\delta_{1,j_1}^{T_\delta}|,\ldots,|\delta_{k,j_k}^{T_{\delta}}|$ are $k$ non-trivial coprime subdegrees of $T$ acting on $\Delta$. 
\end{proof}

\begin{proof}[Proof of Theorem~\ref{thm4}]Assume that
Theorem~\ref{conj3} holds true. Let $G$ be a primitive permutation group on
$\Omega$ with three non-trivial coprime subdegrees. From
Theorems~\ref{thm3} and~\ref{TW}, $G$ is of AS or PA type. Since
Theorem~\ref{conj3} holds true, Proposition~\ref{propo} yields a
contradiction.
\end{proof}

\section{Proofs of Theorems~\ref{cor:3} and~\ref{cor:4}}\label{CorProof}

As usual, we denote by $\F G$ the \emph{Fitting subgroup} of the finite group $G$, that is, the largest normal nilpotent subgroup of $G$. The proof of Corollary~\ref{cor:3} and~\ref{cor:4} will follow from Lemma~\ref{apeman} and from the results in Section~\ref{sec4}.

\begin{lemma}\label{apeman}
Let $H$ be a finite permutation group. Let $O_1,\ldots,O_t$ be $H$-orbits having pairwise coprime size, with $|O_i|>1$ and with $H$ faithful on $O_i$ for each $i\in \{1,\ldots,t\}$. Then $t \leq 2$. Moreover, if $\F H \neq 1$, then $t = 1$
\end{lemma}
\begin{proof}
For each $i\in \{1,\ldots,t\}$, let $\omega_i$ be an element of $O_i$ and set $H_i=H_{\omega_i}$. By hypothesis, $H_i$ is a proper core-free subgroup of $H$.
Let $N$ be a minimal normal subgroup of $H$ and set $N_i=H_i\cap N$. As $H_i$ is core-free in $H$, we have $N_i\subsetneq N$. Note that $|H_iN:H_i|=|N:(H_i\cap N)|=|N:N_i|$ and hence $|H : H_i |$ is a multiple of $|N : N_i |$. This shows that $\{N_i\}_{i=1,\ldots,t}$ is a family of proper subgroups of $N$ with $|N:N_i|$ relatively prime to $|N:N_j|$, for each two distinct elements $i$ and $j$ in $\{1,\ldots,t\}$. Hence $t\leq \mu(N)$. If $N$ is a $p$-group (for some prime $p$) then $\mu(N)=1$ and if $N$ is non-abelian then $\mu(N ) \leq 2$
by Lemma~\ref{id1}.
\end{proof}

\begin{proof}[Proof of Theorem~\ref{cor:3}]
Let $G$ be a non-regular finite transitive permutation group on $\Omega$ and let $\alpha$ be in $\Omega$. Set $H=G_\alpha$.   The result now follows from Lemma~\ref{apeman}.
\end{proof}

\begin{proof}[Proof of Theorem~\ref{cor:4}]
If $t=1$, then there is nothing to prove. So we may assume that $t\geq 2$. We claim that $K/k$ is a separable extension (and so, as $K/k$ is normal, a Galois extension). This is clear if $k$ has characteristic $0$. Suppose then that $k$ has characteristic $p>0$. Now, if for some $i\in \{1,\ldots,t\}$, the extension $k_i/k$ is separable, then $K/k$ is separable (being the normal closure of a separable extension). Therefore we may assume that $k_i/k$ is inseparable, for each $i\in \{1,\ldots,t\}$. This gives that $p$ divides $[k_i:k]$, for each $i\in \{1,\ldots,t\}$. As $t\geq 2$, we obtain a  contradiction and the claim is proved.

Let $H$ be the Galois group $\Gal(K/k)$ and set $H_i = \Gal(K/k_i )$, for $i\in \{1,\ldots,t\}$. Since the normal closure of $k_i$ is $K$, we obtain that $H_i$ is core-free in $H$. Therefore $H_1,\ldots,H_t$ is a family of core-free subgroups of $H$ of pairwise coprime index. Now apply Lemma~\ref{apeman} to obtain $t\leq 2$.
\end{proof}

\section{Coprime factorizations of non-abelian simple groups}\label{sectable}

Liebeck, Praeger and Saxl~\cite{LPS,LPS2} completely classified the maximal
factorizations of all finite almost simple groups. Tables~1--6 and
Theorem~D of \cite{LPS} determine all the triples $(G,A,B)$ 
where $G$ is a nonabelian simple group, and $A$ and $B$ are maximal
subgroups of $G$ with $G=AB$.

Now, if $A'$ and $B'$ are subgroups of $G$ with $|G:A'|$ relatively
prime to $|G:B'|$, then $G=A'B'$. In particular, $A'$ and $B'$ give
rise to a coprime factorization of $G$. 
Moreover, if $A$ (respectively
$B$) is a maximal subgroup of $G$ with $A'\subseteq A$ (respectively
$B'\subseteq B$), then $G=AB$ is a maximal coprime factorization.
Therefore, the list of all non-abelian simple groups 
admitting a \emph{coprime factorization} can be easily obtained with
some elementary 
arithmetic from~\cite{LPS}. Namely, for each triple $(G,A,B)$ in
Tables~1--6 and in Theorem~D of~\cite{LPS}, it suffices to check
whether $|G:A|$ is 
relatively prime to $|G:B|$. Table~\ref{table} in this paper gives all possible
maximal coprime factorizations $(G,A,B)$ and, in particular, the list
of all non-abelian simple groups admitting a coprime
factorization. The notation we use is standard and follows the 
notation in~\cite[Section~1.2]{LPS}.

\begin{table}[!h]
\begin{tabular}{|c|c|c|l|}\hline
$G$ &$A$ & $B$&Remark\\\hline
$\Alt(p^{\ell})$&$\Alt(p^\ell-1)$&max. imprimitive&$p$ prime, $\ell\geq 2$, $i\in
  \{1,\ldots,\ell-1\}$\\
$\Alt(p)$&$\Alt(p-1)$&max. primitive&$p$ prime\\
$\Alt(8)$&$\Alt(8)_x$&\textrm{AGL}(3,2)&$x\subseteq \{1,\ldots,8\},
  |x|\in \{1,2,3\}$\\
\hline  
$M_{11}$&$L_2(11)$&$M_{10}$&\\
$M_{11}$&$L_2(11)$&$M_{9}.2$&\\
$M_{23}$&$23:11$  &$M_{22}$&\\ 
$M_{23}$&$23:11$  &$M_{21}.2$&\\ 
$M_{23}$&$23:11$  &$2^4.\Alt(7)$&\\ 
$M_{24}$  &$M_{23}$&$2^6.3.\Sym(6)$&\\\hline
$L_4(q)$&$\mathrm{PSp}(4,q)$ &$P_i$&$i\in\{1,3\}$, $q$ odd, $q\not\equiv 1\mod 8$\\
$L_n(q)$&$\mathrm{PSp}(n,q)$ &$P_i$&$i\in\{1,n-1\}$, $n=2^r$, $r\geq 2$, $q$ even\\
$L_n(q)$&   ${^{\widehat{}}}\mathrm{GL}_{b^{r-1}}(q^b).b$            & $P_i$                  &$i\in\{1, n-1\}$, $n=b^r$, $r\geq 1$, $b$ prime, \\
 &                    &    &               $r=1$ if $b=2$ and $q\equiv 3\mod 4$,\\
&                   &                   &$b>2$ if $q\equiv 1\mod 4$ \\
$L_2(q)$ &$P_1$&$\Sym(4)$&$q\in \{7,23\}$\\
$L_2(q)$&$P_1$    &$\Alt(5)$&$q\in\{11,19,29,59\}$\\
$L_5(2)$&$P_i$     &$31:5$&$i\in\{2,3\}$\\\hline
$U_{2^r}(2^k)$&$N_1$&$P_{2^{r-1}}$&$r\geq 2$, $k\geq 1$\\
$U_4(2)$           & $3^3.\Sym(4)$&$P_2$&\\\hline
$\mathrm{PSp}_{2m}(q)$& $O_{2m}^{-}(q)$&$P_m$&$m$ \textrm{ odd and} $q$ \textrm{even}\\
$\mathrm{PSp}_4(3)$& $2^4.\mathrm{Alt}(5)$&$P_i$&$i\in \{1,2\}$\\\hline
$\Omega_{2m+1}(q)$&$N_1^-$&$P_m$&$q$ odd, $m\geq 3$ odd\\
$\Omega_{2m}^+(q)$&$N_1$&$P_i$&$i\in \{m-1,m\}$, $m\geq 5$ odd,\\
&&&  $q$ even or $q\equiv 3\mod 4$\\ 
$\Omega_7(3)$&$\mathrm{Sp}_6(2)$&$P_3$&\\
$\Omega_7(3)$            &$2^6.\Alt(7)$  &$P_3$&\\
$\mathrm{P}\Omega_8^+(3)$&$\Omega_8^+(2)$&$P_i$&$i\in\{1,3,4\}$\\\hline
\end{tabular}
\caption{Maximal coprime factorizations of a finite non-abelian simple
group $G$}\label{table}
\end{table}

\end{document}